\documentclass[12pt,oneside]{amsart}
\pdfoutput=1

\usepackage{amssymb,amsxtra}

\usepackage[
  paper=a4paper,
  headsep=15pt,
  text={160mm,235mm},
  centering,
  includehead
]{geometry}

\usepackage{newtxtext,newtxmath}

\usepackage[bookmarks,pdfencoding=auto]{hyperref}
\usepackage{graphicx,xcolor}
\usepackage{enumitem}\setlist[enumerate,1]{font=\upshape}
\usepackage{tikz}
\usetikzlibrary{cd,calc,arrows}
\usepackage[mode=buildnew]{standalone}

\iftrue
  \makeatletter
  \def\@settitle{%
    \vspace*{-10pt}
    \begin{flushleft}%
      \LARGE\bfseries
      \strut\@title\strut
    \end{flushleft}%
  }
  \def\@setauthors{%
    \begingroup
    \def\thanks{\protect\thanks@warning}%
    \trivlist
    \raggedright
    \large \@topsep27\p@\relax
    \advance\@topsep by -\baselineskip
  \item\relax
    \author@andify\authors
    \def\\{\protect\linebreak}%
    \authors
    \ifx\@empty\contribs
    \else
    ,\penalty-3 \space \@setcontribs
    \@closetoccontribs
    \fi
    \normalfont
    \endtrivlist
    \endgroup
  }
  \def\@setaddresses{\par
    \nobreak \begingroup
    \small\raggedright
    \def\author##1{\nobreak\addvspace\smallskipamount}%
    \def\\{\unskip, \ignorespaces}%
    \interlinepenalty\@M
    \def\address##1##2{\begingroup
      \par\addvspace\bigskipamount\noindent
      \@ifnotempty{##1}{(\ignorespaces##1\unskip) }%
      {\ignorespaces##2}\par\endgroup}%
    \def\curraddr##1##2{\begingroup
      \@ifnotempty{##2}{\nobreak\noindent\curraddrname
        \@ifnotempty{##1}{, \ignorespaces##1\unskip}\/:\space
        ##2\par}\endgroup}%
    \def\email##1##2{\begingroup
      \@ifnotempty{##2}{\nobreak\noindent E-mail address%
        \@ifnotempty{##1}{, \ignorespaces##1\unskip}\/:\space
        \ttfamily##2\par}\endgroup}%
    \def\urladdr##1##2{\begingroup
      \def~{\char`\~}%
      \@ifnotempty{##2}{\nobreak\noindent\urladdrname
        \@ifnotempty{##1}{, \ignorespaces##1\unskip}\/:\space
        \ttfamily##2\par}\endgroup}%
    \addresses
    \endgroup
    \global\let\addresses=\@empty
  }
  \def\@setabstracta{%
    \ifvoid\abstractbox
    \else
    \skip@30pt \advance\skip@-\lastskip
    \advance\skip@-\baselineskip \vskip\skip@
    \box\abstractbox
    \prevdepth\z@ 
    \vskip-18pt
    \fi
  }
  \renewenvironment{abstract}{%
    \ifx\maketitle\relax
    \ClassWarning{\@classname}{Abstract should precede
      \protect\maketitle\space in AMS document classes; reported}%
    \fi
    \global\setbox\abstractbox=\vtop \bgroup
    \normalfont\small
    \list{}{\labelwidth\z@
      \leftmargin0pc \rightmargin\leftmargin
      \listparindent\normalparindent \itemindent\z@
      \parsep\z@ \@plus\p@
      
    }%
  \item[\hskip\labelsep\bfseries\abstractname.]%
  }{%
    \endlist\egroup
    \ifx\@setabstract\relax \@setabstracta \fi
  }

  \def\ps@headings{\ps@empty
    \def\@evenhead{%
      \setTrue{runhead}%
      \normalfont\scriptsize
      \rlap{\thepage}\hfill
      \def\thanks{\protect\thanks@warning}%
      \leftmark{}{}}%
    \def\@oddhead{%
      \setTrue{runhead}%
      \normalfont\scriptsize
      \def\thanks{\protect\thanks@warning}%
      \rightmark{}{}\hfill \llap{\thepage}}%
    \let\@mkboth\markboth
  }\ps@headings

  \def\section{\@startsection{section}{1}%
    \z@{-1.4\linespacing\@plus-.5\linespacing}{.8\linespacing}%
    {\normalfont\bfseries\Large}}
  \def\subsection{\@startsection{subsection}{2}%
    \z@{-.8\linespacing\@plus-.3\linespacing}{.5\linespacing\@plus.2\linespacing}%
    {\normalfont\bfseries\large}}
  \def\subsubsection{\@startsection{subsubsection}{3}%
    \z@{.7\linespacing\@plus.2\linespacing}{-1.5ex}%
    {\normalfont\itshape}}
  \def\paragraph{\@startsection{paragraph}{4}%
    \z@{.7\linespacing\@plus.2\linespacing}{-1.5ex}%
    {\normalfont\itshape}}
  \def\@secnumfont{\bfseries}

  \renewcommand\contentsnamefont{\bfseries}
  \def\@starttoc#1#2{\begingroup
    \setTrue{#1}%
    \par\removelastskip\vskip\z@skip
    \@startsection{}\@M\z@{\linespacing\@plus\linespacing}%
    {.5\linespacing}{
      \contentsnamefont}{#2}%
    \ifx\contentsname#2%
    \else \addcontentsline{toc}{section}{#2}\fi
    \makeatletter
    \@input{\jobname.#1}%
    \if@filesw
    \@xp\newwrite\csname tf@#1\endcsname
    \immediate\@xp\openout\csname tf@#1\endcsname \jobname.#1\relax
    \fi
    \global\@nobreakfalse \endgroup
    \addvspace{32\p@\@plus14\p@}%
    \let\tableofcontents\rela\x
  }
  \def\contentsname{Contents}
  \def\l@section{\@tocline{2}{.5ex}{0mm}{5pc}{}}
  \def\l@subsection{\@tocline{2}{0pt}{2em}{5pc}{}}
  \makeatother

\fi

\def\to{\mathchoice{\longrightarrow}{\rightarrow}{\rightarrow}{\rightarrow}}
\makeatletter
\newcommand{\shortxra}[2][]{\ext@arrow 0359\rightarrowfill@{#1}{#2}}
\def\longrightarrowfill@{\arrowfill@\relbar\relbar\longrightarrow}
\newcommand{\longxra}[2][]{\ext@arrow 0359\longrightarrowfill@{#1}{#2}}

\makeatother

\makeatletter
\def\addtagsub#1{\let\oldtf=\tagform@\def\tagform@##1{\oldtf{##1}\hbox{$_{#1}$}}}
\makeatother

\makeatletter
\def\Nopagebreak{\@nobreaktrue\nopagebreak}
\makeatother

\foreach \n in 
{Z,Q,R,C}
{\expandafter\xdef\csname\n\endcsname{\noexpand\mathbb{\n}}}

\foreach \n in 
{A,F}
{\expandafter\xdef\csname b\n\endcsname{\noexpand\mathbb{\n}}}

\foreach \n in 
{A,...,Z}
{\expandafter\xdef\csname c\n\endcsname{\noexpand\mathcal{\n}}}

\foreach \n in 
{dim,sign,sgn,Ker,Coker,Im,Hom,End,GL,SO,
 Tor,Ext,Gal,Irr,dis,Wh,tr,Tr,vol,colim,lk,st,Mod,Aut,SL}
{\expandafter\xdef\csname\n\endcsname{\noexpand\operatorname{\n}}}

\foreach \n/\s in 
{inte/int}
{\expandafter\xdef\csname\n\endcsname{\noexpand\operatorname{\s}}}

\foreach \n in 
{id}
{\expandafter\xdef\csname\n\endcsname{\noexpand\textrm{\n}}}

\foreach \n/\s in 
{csum/\#,bcsum/\natural/}
{\expandafter\xdef\csname\n\endcsname{\noexpand\mathbin{\s}}}

\def\sm{\smallsetminus}
\def\cupover#1{\mathbin{\mathop{\cup}\limits_{#1}}}
\def\cuptover#1{\mathbin{\mathop{\cup}\nolimits_{#1}}}

\makeatletter
\newtheoremstyle{theorem-giventitle}
        {}{}              
        {\itshape}                      
        {}                              
        {\bfseries}                     
        {.}                             
        {\thm@headsep}                             
        {\thmnote{\bfseries#3}}
\newtheoremstyle{theorem-givenlabel}
        {}{}              
        {\itshape}                      
        {}                              
        {\bfseries}                     
        {.}                             
        {\thm@headsep}                             
        {\thmname{#1}~\thmnumber{#3}\setcurrentlabel{#3}}
\newtheoremstyle{definition-giventitle}
        {}{}              
        {}                      
        {}                              
        {\bfseries}                     
        {.}                             
        {\thm@headsep}                             
        {\thmnote{\bfseries#3}}
\def\setcurrentlabel#1{\gdef\@currentlabel{#1}}
\makeatother

\newtheorem{theorem}{Theorem}[section]
\newtheorem{theoremalpha}{Theorem}

\newtheorem{proposition}[theorem]{Proposition}

\newtheorem{lemma}[theorem]{Lemma}

\theoremstyle{definition}
\newtheorem{definition}[theorem]{Definition}
\newtheorem{example}[theorem]{Example}

\newtheorem*{convention}{Convention}

\theoremstyle{theorem-giventitle}
\newtheorem{theorem-named}{}
\theoremstyle{theorem-givenlabel}
\newtheorem{theorem-labeled}{Theorem}

\theoremstyle{definition-giventitle}
\newtheorem{definition-named}{}
\newtheorem{conjecture-named}{}
\newtheorem{case-named}{}

\numberwithin{equation}{section}

\begin{document}

\title
{Calegari's homotopy 4-spheres from fibered knots are standard}

\author{Jae Choon Cha}
\address{
  Center for Research in Topology and Department of Mathematics\\
  POSTECH\\
  Pohang 37673\\
  Republic of Korea
}
\email{jccha@postech.ac.kr}

\author{Min Hoon Kim}
\address{
  Department of Mathematics\\
  Ewha Womans University\\
  Seoul 03760\\
  Republic of Korea
}\email{minhoonkim@ewha.ac.kr}

\thanks{%
  The first author was partly supported by the National Research Foundation grant 2019R1A3B2067839.
  The second author was partly supported by Samsung Science and Technology Foundation (SSTF-BA2202-01) and the National Research Foundation grant 2021R1C1C1012939.
}


\begin{abstract}
  In 2009, Calegari constructed smooth homotopy 4-spheres from monodromies of fibered knots.
  We prove that all these are diffeomorphic to the standard 4-sphere.
  Our method uses 5-dimensional handlebody techniques and results on mapping class groups of 3-dimensional handlebodies.
  As an application, we present potential counterexamples to the smooth 4-dimensional Schoenflies conjecture which are related to the work of Casson and Gordon on fibered ribbon knots.
\end{abstract}

\maketitle

\section{Introduction}

Concerning the smooth 4-dimensional Poincar\'e conjecture, Calegari introduced a construction of smooth homotopy 4-spheres from fibered knots in~$S^3$~\cite{Calegari:2009-1}.
Briefly, the construction is as follows.
For a given fibered knot of genus~$g$, the monodromy on the fiber surface induces an automorphism of the fundamental group, which is the free group of rank $n=2g$.
We say that such an automorphism is \emph{geometric}.
Choosing an orientation preserving diffeomorphism on the connected sum $n(S^1\times S^2)$ of $n$ copies of $S^1\times S^2$ that realizes a geometric automorphism, the associated Calegari sphere is obtained from the mapping torus of the diffeomorphism by surgery along a circle that is transverse to each fiber.
See Definition~\ref{definition:calegari-sphere} and Proposition~\ref{proposition:calegari-is-homotopy-sphere} for details.

Observing that known techniques in the literature that produce a diffeomorphism to $S^4$ are unlikely to work for his homotopy spheres in general, Calegari asked whether these are diffeomorphic to $S^4$ or not.
See below for a related discussion.

In this paper, we answer Calegari's question.

\begin{theoremalpha}
  \label{theorem:main}
  All Calegari homotopy 4-spheres from fibered knots are diffeomorphic to~$S^4$.
\end{theoremalpha}

In the remaining part of this introduction, we discuss some background, our methods and applications.
Throughout this paper, we work in the smooth category.

\subsubsection*{Comparison with Cappell-Shaneson spheres}

It is noteworthy that Calegari was inspired by the Cappell-Shaneson homotopy 4-spheres, which are obtained by a similar construction using mapping tori with 3-torus fiber instead of $n(S^1\times S^2)$.
The question of whether all Cappell-Shaneson spheres are standard is still open, although there have been remarkable progresses toward an affirmative answer (e.g.,\ \cite{Akbulut-Kirby:1979-1,Aitchison-Rubinstein:1984-1,Akbulut-Kirby:1985-1,Gompf:1991-1,Gompf:1991-2,Akbulut:2010-1,Gompf:2010-1}).
Regarding monodromies and handle decompositions, the Calegari spheres seem more complicated to deal with than the Cappell-Shaneson spheres.
While the monodromies of the Cappell-Shaneson spheres are represented by specific matrices in $\SL(3,\Z)$ indexed by three integers~\cite{Aitchison-Rubinstein:1984-1}, the monodromies of Calegari spheres are from a large collection of free group automorphisms that do not have such a simple parametrization.
Known handle decompositions of the Cappell-Shaneson spheres consist of a fixed number of handles, but in a natural handle decomposition of a Calegari sphere with fiber $n(S^1\times S^2)$, the number of handles grows linearly in~$n$, resulting in difficulties in drawing diagrams and applying handle calculus techniques.

\subsubsection*{Methods of the proof}

Instead of handle diagrams, we use a 5-dimensional approach.
We express a Calegari sphere as the boundary of a certain 5-dimensional handlebody with handles of index${}\le 2$, by investigating the structure of the mapping class of the monodromy.
The algebraic property of the involved free group automorphism ensures that the handlebody is contractible.
In fact, the handlebody has the homotopy type of a contractible 2-complex associated with a balanced presentation of a trivial group.
One would attempt to simplify the handle structure by handle slides and eliminations of $(1,2)$-handle pairs to reach a trivial handle decomposition, but it is well-known that the Andrews-Curtis problem is an obstacle.
Successful methods which resolved this issue in the literature were essentially ad-hoc handle calculus that introduces a canceling $(2,3)$-handle pair.

We use a different method.
The open book decomposition of $S^3$ associated with the given fibered knot leads to a handle decomposition of the 3-ball $D^3$, which algebraically resembles the handle decomposition of the 5-dimensional handlebody associated with the Calegari sphere.
Taking the product with $D^2$, we obtain a handle decomposition of the standard 5-ball.
We compare it with the concerned 5-dimensional handlebody.
Since they are built from the same algebraic data, it turns out that they have identical attaching circles of 2-handles in this dimension, so the attaching framing is the remaining issue.
It follows that they differ by Gluck twists on the boundary, and we complete the proof by verifying that the involved Gluck twists do not change the diffeomorphism type.
For the full details, see Section~\ref{section:proof-main}.

We remark that the construction of Calegari spheres from geometric automorphisms generalizes to a larger class of free group automorphisms.
See Section~\ref{section:calegari-construction} for details.
In the non-geometric case, our method outlined above does not apply directly.
It remains open whether all homotopy 4-spheres arising from these non-geometric automorphisms are standard.

\subsubsection*{Application to Casson-Gordon balls and the Schoenflies conjecture}

In \cite[Theorem~5.1]{Casson-Gordon:1983-1}, Casson and Gordon showed that a fibered knot $K$ of genus $g$ in a homology 3-sphere $Y$ is homotopy ribbon in a homology 4-ball if and only if the monodromy of $K$ on the capped-off minimal Seifert surface extends to a diffeomorphism $h$ on the 3-dimensional handlebody~$H$ of genus~$g$.
In this case, $K$ bounds a fibered homotopy ribbon 2-disk $D$ in a homology 4-ball~$\Sigma(h)$ bounded by~$Y$, where $\Sigma(h)$ is the underlying 4-manifold of the open book decomposition associated with the monodromy~$h$, having binding $D$ and page~$H$.
We call $\Sigma(h)$ a \emph{Casson-Gordon (homology) ball}.
See Definition~\ref{definition:casson-gordon-ball} for a detailed description of~$\Sigma(h)$.
When $Y=S^3$, it is known that $\Sigma(h)$ is a homotopy 4-ball, i.e., contractible and bounded by~$S^3$, and an open problem posed by Casson and Gordon is whether $\Sigma(h)$ is diffeomorphic to~$D^4$.
A positive resolution would imply that a fibered knot $K$ in~$S^3$ is homotopy ribbon in $D^4$ if and only if its monodromy extends to a handlebody.
There are interesting recent related results, e.g., see \cite{Miller:2021-1} and \cite{Meier-Zupan:2022-1}.

As an application of Theorem~\ref{theorem:main}, we show the following result which concerns the double of~$\Sigma(h)$.
We say that a homotopy 4-ball is a \emph{Schoenflies ball} if it embeds in~$S^4$.

\begin{theoremalpha}
  \label{theorem:casson-gordon-ball}
  If an orientation preserving diffeomorphism $h$ on the 3-dimensional handlebody $H_{2k}$ of genus~$2k$ induces a geometric automorphism of the free group $\pi_1(H_{2k})$, then the double of the Casson-Gordon ball $\Sigma(h)$ is diffeomorphic to~$S^4$.
  Consequently,
  the homology 3-sphere $\partial\Sigma(h)$ embeds in~$S^4$;
  and if $\partial\Sigma(h)=S^3$, then $\Sigma(h)$ is a Schoenflies ball.
\end{theoremalpha}

\begin{proof}[Proof of Theorem~\ref{theorem:casson-gordon-ball}]
  In Lemmas~\ref{lemma:calegari-double-decomposition} and~\ref{lemma:monodromy-double-decomposition} below, we show that a homotopy 4-sphere is a Calegari sphere associated with a diffeomorphism on $n(S^1\times S^2)$ inducing $\phi\in \Aut(F_n)$ if and only if $X$ is the double of a contractible Casson-Gordon ball associated with a diffeomorphism on $H_n$ inducing the same~$\phi$.
  Thus, Theorem~\ref{theorem:casson-gordon-ball} follows from Theorem~\ref{theorem:main}\@.
\end{proof}

We remark that certain special cases of the Schoenflies balls $\Sigma(h)$ in Theorem~\ref{theorem:casson-gordon-ball} with $k=1$ (i.e.\ genus${}=2$) are known to be $D^4$~\cite{Aitchison:1984-1,Hitt-Silver:1991-1} (see also related work in \cite{Quach-Weber:1979-1,Kanenobu:1981-1, Kanenobu:1986-1}).
Beyond these, Theorem~\ref{theorem:casson-gordon-ball} appears to produce new Schoenflies balls $\Sigma(h)$ from a variety of handlebody monodromies $h$ (especially on handlebodies of larger genera).
We remark that Gabai, Naylor and Schwartz recently obtained a family of Schoenflies balls using Gluck twists on certain knotted 2-spheres in~$S^4$~\cite{Gabai-Naylor-Schwartz:2023-1}.  

Theorem~\ref{theorem:casson-gordon-ball} also has a potential application, which is related to the study of the smooth 4-dimensional Poincar\'e conjecture. Freedman suggested the following approach~\cite{Freedman:2024-1}:
for a contractible 4-manifold $X$, $\partial X$ is a homology 3-sphere, and if one could prove that $\partial X$ does not smoothly embed in~$S^4$, then the double of $X$ would be a homotopy 4-sphere not diffeomorphic to~$S^4$, providing a counterexample.
Theorem~\ref{theorem:casson-gordon-ball} may be useful in verifying that homology 3-spheres of interest do embed in~$S^4$.
We intend to investigate this elsewhere.

The remainder of this paper is organized as follows.
In Section~\ref{section:calegari-construction}, we describe Calegari's construction of homotopy 4-spheres and make some relevant observations.
In Section~\ref{section:proof-main}, we prove Theorem~\ref{theorem:main}.

\section{Calegari's constuction of homotopy 4-spheres}
\label{section:calegari-construction}

We begin with the definition of a Calegari sphere, following~\cite{Calegari:2009-1}.
Let $M_n$ be the 3-manifold obtained from the connected sum $n(S^1\times S^2)$ of $n$ copies of $S^1\times S^2$ by removing an open 3-ball.
Let $F_n$ be the free group of rank~$n$, generated by $x_1,\ldots,x_n$.
Identify $\pi_1(M_n)$ with $F_n$ using (a point in) $\partial M_n\cong S^2$ as a basepoint.

\begin{convention}
  Throughout this paper, when we say a self-map on $X$ fixes a subset $A\subset X$, it means that $A$ is fixed pointwise.
\end{convention}

For a diffeomorphism $f\colon M_n \to M_n$ that fixes $\partial M_n$, define the \emph{relative mapping torus of~$f$} by attaching $\partial M_n \times D^2$ to the mapping torus $T(f) = M_n \times [0,1] / (x,0) \sim (f(x),1)$ along $\partial T(f) = \partial M_n\times S^1$.
We denote it by
\[
  \Sigma(f) = T(f) \cupover{\partial M_n\times S^1} \partial M_n \times D^2.
\]
Note that $\Sigma(f)$ is homeomorphic to $T(f)/(x,t) \sim (x,t'),\, x\in \partial M_n,\, t\in S^1$. 
So, $\Sigma(f)$ has an open book decomposition with page $M_n$ and binding $\partial M_n = \partial M_n\times 0 \subset \partial M_n\times D^2$.

\begin{definition}[Calegari spheres \cite{Calegari:2009-1}]
  \label{definition:calegari-sphere}
  Let $f$ be a diffeomorphism of $M_n$ that fixes~$\partial M_n$.
  Let $\phi\in \Aut(F_n)$ be the automorphism induced by~$f$.
  If the presentation
  \begin{equation}
    \langle x_1,\ldots,x_n\mid \phi (x_i)=x_i,\;i=1,\ldots,n\rangle
    \tag{$P_\phi$}
    \label{equation:calegari-presentation}
  \end{equation}
  defines a trivial group, then the 4-manifold $\Sigma(f)$ is called a \emph{Calegari sphere}.
\end{definition}

There are many automorphisms $\phi\in \Aut(F_n)$ for which the presentation~\eqref{equation:calegari-presentation} is a trivial group.
The example below is the main focus of Calegari's article~\cite{Calegari:2009-1} and of this paper.

\begin{definition}
  \label{definition:geometric-automorphism}
  An automorphism $\phi$ of a free group is a \emph{geometric automorphism} if $\phi$ is induced by the monodromy on the fiber of a fibered knot $K$ in~$S^3$.
\end{definition}

\begin{example}[Homotopy 4-spheres from fibered knots in $S^3$~\cite{Calegari:2009-1}]
  \label{example:calegari-spheres-from-fibred-knots}
  For a geometric automorphism $\phi$ on $F_{2g}$, the presentation \eqref{equation:calegari-presentation} gives the fundamental group of the ambient space $S^3$ of the fibered knot $K$, and thus it is a trivial group.
  We will show, in Lemma~\ref{lemma:monodromy-realization} below, that there are exactly $2^{2g}$ diffeomorphisms $f$ on $M_{2g}$ that fix $\partial M_{2g}$ and induce~$\phi$ (up to isotopy fixing $\partial M_{2g}$).
  Each of these $f$ gives a Calegari sphere~$\Sigma(f)$.
\end{example}

There are non-geometric automorphisms $\phi$ for which the presentation~\eqref{equation:calegari-presentation} is a trivial group.
Definition~\ref{definition:calegari-sphere} and the observations in this section allow such $\phi$ and are not limited to geometric automorphisms.
We hope this is useful for future studies.

Specifically, the example below illustrates that Calegari spheres also arise from automorphisms of a free group of odd rank, while a geometric automorphism is always on a free group of even rank.

\begin{example}
  \label{example:fibered-ribbon-knot}
  There are fibered knots $K$ in $S^3$ whose monodromy on the capped-off minimal Seifert surface extends to a 3-dimensional handlebody $H$ of the same genus as~$K$.
  In fact, by work of Casson and Gordon~\cite{Casson-Gordon:1983-1} (combined with Freedman's work~\cite{Freedman:1982-1}), a fibered knot $K$ in $S^3$ satisfies this if and only if $K$ is topologically homotopy ribbon, that is, $K$ bounds a locally flat disk $\Delta$ topologically embedded in $D^4$ such that $\pi_1(S^3 \sm K) \to \pi_1(D^4 \sm \Delta)$ is surjective.
  When it holds for $K$, let $\phi\in \Aut(F_g)$ be the automorphism induced by an extended monodromy $h\colon H\to H$ on the handlebody~$H$.
  Then the presentation \eqref{equation:calegari-presentation} associated with $\phi$ is a trivial group, since it defines a quotient of the presentation associated with the monodromy of the fibered knot $K$ which presents $\pi_1(S^3)=0$.
  We will see, using Lemma~\ref{lemma:monodromy-realization} below, that $\phi$ is induced by exactly $2^g$ diffeomorphisms $f$ on $M_g$ each of which produces a Calegari sphere~$\Sigma(f)$.

  Also, note that the double of the 3-dimensional handlebody $H$ of genus~$g$ is the connected sum $g(S^1\times S^2)$, and the double of the extended monodromy $h$ on $H$ induces a diffeomorphism on $M_g = g(S^1\times S^2) \sm \inte D^3$ which realizes the automorphism~$\phi$.
  This may be viewed as a precursor of Lemmas~\ref{lemma:calegari-double-decomposition} and~\ref{lemma:monodromy-double-decomposition} in Subsection~\ref{subsection:double-decomposition}, which deals with a more general case.
\end{example}

Calegari observed the following:

\begin{proposition}[\cite{Calegari:2009-1}]
  \label{proposition:calegari-is-homotopy-sphere}
  A Calegari sphere $\Sigma(f)$ is a homotopy 4-sphere.
\end{proposition}

For the reader's convenience, we describe a proof.

\begin{proof}
  The presentation~\eqref{equation:calegari-presentation} gives $\pi_1(\Sigma(f))$, and thus $\Sigma(f)$ is simply connected.
  So it remains to show $H_2(\Sigma(f))=0$.
  Since $\partial M_n=S^2$, we have $H^1(M_n,\partial M_n) = H^1(M_n)$.
  Thus $H_2(M_n) = H^1(M_n,\partial M_n) = \Hom(H_1(M_n),\Z)$ by the Poincar\'e duality and the universal coefficient theorem.
  So, if $f$ induces an $n\times n$ matrix $A$ on $H_1(M_n) = \Z^n$, then $f$ induces $A^t$ on $H_2(M_n)$.
  Since \eqref{equation:calegari-presentation} is a trivial group, $A-I$ is invertible, and consequently so is $A^t-I$.
  For the mapping torus $T=T(f)$ of $f$, $H_2(T)$ is presented by $A^t-I$, and thus $H_2(T)=0$.
  From the Mayer-Vietoris sequence for $\Sigma(f) = T \cuptover{S^2\times S^1} S^2\times D^2$, one readily verifies that $H_2(\Sigma(f)) = H_2(T)$.
\end{proof}

Let $\Mod(M_n,\partial M_n)$ be the mapping class group of diffeomorphisms that fix~$\partial M_n$.
Let $\rho\colon \Mod(M_n,\partial M_n) \to \Aut(F_n)$ be the homomorphism sending a mapping class to the induced map on $\pi_1(M_n,\partial M_n) = F_n$.
An elementary argument (using handle slides) shows that every elementary Nielsen transformation on the free group $F_n$ is induced by a diffeomorphism $M_n \to M_n$ that fixes $\partial M_n$, and thus so is an arbitrary automorphism on~$F_n$.
That is, $\rho$ is surjective.
Furthermore, the kernel is given as follows:

\begin{lemma}
  \label{lemma:monodromy-realization}
  The homomorphism $\rho\colon \Mod(M_n,\partial M_n) \to \Aut(F_n)$ is surjective and has kernel isomorphic to~$(\Z_2)^n$.
  Consequently, every automorphism $\phi\in\Aut(F_n)$ is induced by $2^n$ diffeomorphisms of $M_n$ fixing~$\partial M_n$, up to isotopy fixing~$\partial M_n$.
\end{lemma}

\begin{proof}
  Let $\Mod(n(S^1\times S^2), *)$ be the mapping class group of orientation preserving diffeomorphisms fixing the basepoint~$*$.
  By work of Laudenbach~\cite{ Laudenbach:1973-1,Laudenbach:1974-1}, there is an exact sequence
  \[
    0\to (\Z_2)^n \to \Mod(n(S^1\times S^2), *) \to \Aut(F_n) \to 0
  \]
  where the subgroup $(\Z_2)^n$ is generated by sphere twists along the core spheres of the $n$ summands of $n(S^1\times S^2)$.
  For the reader's convenience, we recall that the sphere twist along an embedded 2-sphere $S^2$ in a 3-manifold is a diffeomorphism supported in a product neighborhood $S^2\times [0,1]$ that rotates the slices $S^2\times t$ according to a loop $\gamma_t \colon [0,1] \to \SO(3)$ that generates $\pi_1(\SO(3))=\Z_2$. (In particular, a sphere twist has order~two.)

  In general, for any closed oriented 3-manifold $N$ with a 3-ball neighborhood $B(*)$ of a basepoint $* \in N$, if we write $M=N\sm \inte B(*)$, the inclusion-induced homomorphism
  \[
    \eta\colon \Mod(M, \partial M) \to \Mod(N, *)
  \]
  is surjective and its kernel is generated by the sphere twist along (a parallel copy of)~$\partial M$.
  Since we did not find an explicit reference, we describe some details:
  by Cerf-Palais (or uniqueness of a tubular neighborhood), an orientation preserving diffeomorphism $f$ on $N$ fixing $*$ is isotopic rel $*$ to a diffeomorphism $f_1$ that restricts to the identity on~$B(*)$.
  The surjectivity of $\eta$ follows from this.
  If a diffeomorphism $f$ on $N$ that restricts to the identity on $B(*)$ is isotopic to the identity rel $*$, then we may assume that the isotopy preserves $B(*)$ setwise.
  The restriction of the isotopy on the 2-sphere $\partial B(*)=\partial M$ defines a loop in $\SO(3)$ and thus induces a twist (or the identity) along the sphere~$\partial M$.
  Its composition with the restriction $f|_M$ is isotopic rel $\partial M$ to the identity on~$M$.
  It follows that the sphere twist along $\partial M$ generates the kernel of~$\eta$.

  In addition, in our special case of $M_n=n(S^1\times S^2) \sm \inte B(*)$, the sphere twist along $\partial M_n$ is trivial in $\Mod(M_n, \partial M_n)$, since it is equal to the composition of the $n$ sphere twists along the connected sum spheres in $M_n$ by~\cite[(I), p.~214]{Hatcher-Wahl:2010-1}, and since the sphere twist along each connected sum sphere is trivial by~\cite[(the last paragraph of) Remark~2.4]{Hatcher-Wahl:2010-1}.
  Therefore, $\Mod(M_n, \partial M_n)$ is isomorphic to $\Mod(n(S^1\times S^2), *)$.
  Combining it with the Laudenbach exact sequence described above, the proof is completed.
\end{proof}

\section{Proof of the main result}
\label{section:proof-main}

In this section, we prove Theorem~\ref{theorem:main}, which asserts that all Calegari spheres arising from geometric automorphisms are diffeomorphic to~$S^4$.
Our proof consists of several steps, each of which we describe in Subsections~\ref{subsection:double-decomposition}--\ref{subsection:comparison-5d-handlebody} below.

Throughout this section, suppose that $\Sigma(f)$ is a Calegari sphere, where $f$ is a diffeomorphism on $M_n$ such that $f$ fixes $\partial M_n$ and the associated presentation~\eqref{equation:calegari-presentation} is a trivial group.
(In Subsections~\ref{subsection:double-decomposition} and~\ref{subsection:handle-decomposition-contractible-casson-gordon} below, we do not assume that $f$ is geometric.)

\subsection{Calegari spheres are doubles}
\label{subsection:double-decomposition}

As the first step, we will observe that every Calegari sphere can be expressed as the double of a contractible 4-manifold.
To state it, we use the following notation.
Let $I=[0,1]$.
Let $H_n$ be a 3-dimensional handlebody of genus $n$, i.e., the boundary connected sum of $n$ copies of $S^1\times D^2$.
Fix a 2-disk embedded in $\partial H_n$ and denote it by $\partial_0 H_n$.
Let $h$ be a diffeomorphism on $H_n$ that fixes~$\partial_0 H_n$.
Here we do not require that $h$ fixes the whole boundary~$\partial H_n$.
Define an associated relative mapping torus $\Sigma(h)=\Sigma(h,\partial_0 H_n)$ by attaching a 2-handle to the mapping torus $T(h) = H_n\times I/(x,0) \sim (h(x),1)$ along the attaching region $\partial_0 H_n \times S^1 \cong D^2\times S^1$:
\[
  \Sigma(h) = T(h) \cupover{\partial_0 H_n \times S^1} \partial_0 H_n \times D^2.
\]
Identifying $\pi_1(H_n) = \pi_1(H_n,\partial_0 H_n)$ with the free group~$F_n$, $h$ induces an automorphism $\phi\in \Aut(F_n)$.
The 4-manifold $\Sigma(h)$ is contractible if and only if the presentation~\eqref{equation:calegari-presentation} associated with $\phi$ is a trivial group, by an argument similar to Lemma~\ref{proposition:calegari-is-homotopy-sphere}, using that $H_2(H_n)=0$.

\begin{definition}[Casson-Gordon ball]
  \label{definition:casson-gordon-ball}
  We call the 4-manifold $\Sigma(h)$ a \emph{Casson-Gordon homology ball} if $h$ is a diffeomorphism on $H_n$ that fixes~$\partial_0 H_n$ such that the associated presentation~\eqref{equation:calegari-presentation} is a perfect group.
  In addition, if the associated presentation~\eqref{equation:calegari-presentation} is a trivial group, we call $\Sigma(h)$ a \emph{Casson-Gordon contractible ball}.
\end{definition}

Note that the boundary of a Casson-Gordon contractible ball is a homology 3-sphere but not necessarily $S^3$ in general.

\begin{lemma}
  \label{lemma:calegari-double-decomposition}
  Every Calegari sphere $\Sigma(f)$ is the double of a Casson-Gordon contractible ball~$\Sigma(h)$.
\end{lemma}

\begin{proof}
  Note that $M_n = H_n \cup -H_n/{\sim}$, where a point $x\in H_n \sm \inte(\partial_0 H_n)$ is identified with $x\in -H_n$.
  Also, a diffeomorphism $h$ on $H_n$ that fixes $\partial_0 H_n$ induces a diffeomorphism $f$ on $M_n$ that fixes $\partial M_n$, and the closed 4-manifold $\Sigma(f)$ is the double of~$\Sigma(h)$.
  (As an abuse of terminology, we could say that $f$ is a ``double'' of~$h$.)
  Thus, Lemma~\ref{lemma:calegari-double-decomposition} follows immediately from Lemma~\ref{lemma:monodromy-double-decomposition} below.
\end{proof}

\begin{lemma}
  \label{lemma:monodromy-double-decomposition}
  A diffeomorphism $f$ on $M_n$ that fixes $\partial M_n$ is isotopic rel $\partial M_n$ to a diffeomorphism induced by a diffeomorphism $h$ on $H_n$ that fixes~$\partial_0 H_n$.
\end{lemma}

Note that $f$ and $h$ induce the same automorphism on the free group~$F_n$.

\begin{proof}
  Let $\phi\in \Aut(F_n)$ be the automorphism induced by the given diffeomorphism $f$ on~$M_n$.
  The map $\Mod(H_n,\partial_0 H_n) \to \Aut(F_n)$ is surjective since elementary Nielsen transformations on $F_n$ are realized by a mapping class using handle slides.
  Choose a diffeomorphism $g$ on $H_n$ that fixes $\partial_0 H_n$ and induces~$\phi$.
  Let $f_0$ be the diffeomorphism on $M_n=H_n \cup -H_n /{\sim}$ induced by~$g$.
  Note that $f_0$ fixes $\partial M_n$ and induces the same automorphism $\phi$ on~$F_n$.
  Thus, by the Laudenbach exact sequence argument in the proof of Lemma~\ref{lemma:monodromy-realization}, it follows that $f$ and $f_0$ differ, in $\Mod(M_n,\partial M_n)$, by a composition of the sphere twists along the $n$ core spheres in $M_n = n(S^2\times S^1) \sm \inte(D^3)$.
  Each of these spheres is the double of a disk embedded in $H_n$, viewing $M_n$ as $H_n \cup -H_n /{\sim}$, and the associated sphere twist is the double of a disk twist, since $\pi_1(\SO(2)) \to \pi_1(\SO(3))$ is surjective.
  Let $h\colon H_n\to H_n$ be the composition of $g$ with the disk twists associated with the involved sphere twists.
  The diffeomorphism $h$ induces, on $M_n$, the composition of $f_0$ and the sphere twists rel $\partial M_n$, and consequently, it represents $[f]$ in $\Mod(M_n,\partial M_n)$.
\end{proof}

In what follows, denote by $h$ a diffeomorphism on the handlebody $H_n$ fixing $\partial_0 H_n$ given by Lemma~\ref{lemma:calegari-double-decomposition}.
That is, the given Calegari sphere $\Sigma(f)$ is the double of the contractible 4-manifold~$\Sigma(h)$.

\subsection{A handle decomposition of the contractible 4-manifold~\texorpdfstring{$\Sigma(h)$}{Σ(h)}}
\label{subsection:handle-decomposition-contractible-casson-gordon}

Construct a handle decomposition $\Sigma(h)$, as described below.
The handlebody $H_n$ consists of one 0-handle and $n$ 1-handles.
Thus the mapping torus $T(h)$ consists of one $0$-handle, $(n+1)$ 1-handles and $n$ 2-handles, where the 2-handles are attached along curves representing the relators of the following presentation of $\pi_1(T(h))$, an HNN extension of $\pi_1(H_n) = F_n$:
\[
  \langle x_1, \ldots ,x_n, t \mid t x_1 t^{-1} = \phi(x_1), \ldots, t x_n t^{-1} = \phi(x_n) \rangle.
\]
Then $\Sigma(h)$ is obtained by attaching to $T(h)$ an additional 2-handle, which is $\partial_0 H_n \times D^2$ in our previous notation;
the attaching circle of this 2-handle represents the generator~$t$.
It follows that the handle structure of $\Sigma(h)$ gives rise to the following presentation of $\pi_1(\Sigma(h))$:
\begin{equation}
  \langle x_1, \ldots ,x_n, t \mid t x_1 t^{-1} = \phi(x_1), \ldots, t x_n t^{-1} = \phi(x_n), \, t \rangle.
  \label{equation:handle-presentation}
\end{equation}

\subsection{A handle decomposition of~\texorpdfstring{$D^3$}{D\textthreesuperior}}
\label{subsection:handle-decomposition-D^3}

From now on, suppose that the automorphism $\phi$ on $F_n$ is geometric, i.e., $\phi$ is induced by the monodromy of a fibered knot $K$ in~$S^3$.
In particular, $n=2g$ where $g$ is the genus of~$K$.

Construct a 3-dimensional relative mapping torus associated with the automorphism~$\phi$, as described below.
Let $S=S_{g,1}$ be an orientable surface of genus $g$ with one boundary component.
Identify $\pi_1(S)$ with the free group $F_{2g}$, using a basepoint in $\partial S$.
Fix an arc $\partial_0 S$ in $\partial S$ containing the basepoint.
Viewing $S$ as the minimal Seifert surface of the fibered knot $K$, the monodromy of $K$ is a diffeomorphism $h_0$ on~$S$ that induces~$\phi$ on~$F_{2g}$.
We may assume that $h_0$ fixes $\partial_0 S$.
Let $\Sigma(h_0)=\Sigma(h_0,\partial_0 S)$ be the relative mapping torus of $h_0$ rel $\partial_0 S$ defined by
\[
  \Sigma(h_0) = \bigl( S \times I / (x,0) \sim (h_0(x),1) \bigr) \cupover{\partial_0 S \times S^1} \partial_0 S \times D^2.
\]
The 3-manifold $\Sigma(h_0)$ is diffeomorphic to the ambient space of $K$ with an open 3-ball removed.
That is, $\Sigma(h_0)$ is diffeomorphic to~$D^3$.

Recall that we have constructed a handle decomposition of the 4-dimensional relative mapping torus~$\Sigma(h)$ in Subsection~\ref{subsection:handle-decomposition-contractible-casson-gordon}.
Apply the same construction to $\Sigma(h_0)$, to obtain a handle decomposition of $\Sigma(h_0)\cong D^3$ which consists of one 0-handle, $(2g+1)$ 1-handles and $(2g+1)$ 2-handles.
Since the homotopy classes of the attaching curves of 2-handles are determined by $\phi$, the handle decomposition of $\Sigma(h_0)$ determines a presentation of the trivial group $\pi_1(\Sigma(h_0))$ identical with~\eqref{equation:handle-presentation}.

\subsection{Comparison of 5-dimensional handlebodies}
\label{subsection:comparison-5d-handlebody}

We have handle decompositions of the 4-manifold $\Sigma(h)$ and the 3-manifold~$\Sigma(h_0)$ which consist of handles of index 0, 1 and~2. 
Consider the 5-dimensional contractible handlebodies $P = \Sigma(h)\times I$ and $Q = \Sigma(h_0) \times I^2$.
Note that $\partial P$ is the double of $\Sigma(h)$, which is diffeomorphic to the given Calegari sphere~$\Sigma(f)$.
Since $\Sigma(h_0)\cong D^3$, $\partial Q = \partial (\Sigma(h_0)\times I^2)$ is diffeomorphic to~$S^4$.

\begin{lemma}
  \label{lemma:5d-handlebody-diffeomorphism}
  The homotopy 4-sphere $\partial P$ is diffeomorphic to~$\partial Q$.
\end{lemma}

From Lemma~\ref{lemma:5d-handlebody-diffeomorphism}, it follows that $\partial P=\Sigma(f)$ is diffeomorphic to~$S^4$.
(In fact, by~\cite[\S 9, Proposition~C]{Milnor:1965-1}, it also follows that $P$ is diffeomorphic to~$D^5$.)
This completes the proof of Theorem~\ref{theorem:main}.

\begin{proof}[Proof of Lemma~\ref{lemma:5d-handlebody-diffeomorphism}]
  The handle decompositions of $\Sigma(h)$ and $\Sigma(h_0)$ induce handle decompositions of $P$ and $Q$, each of which has one 0-handle, $(2g+1)$ 1-handles and $(2g+1)$ 2-handles.
  Let $P_1$ and $Q_1$ be the union of the 0- and 1-handles of $P$ and $Q$, respectively.
  Fix a diffeomorphism $P_1 \cong Q_1$ by identifying the 0- and 1-handles.
  Since the associated presentations of the fundamental group are identical (determined by the automorphism~$\phi$), the attaching circles of the 2-handles are homotopic in the boundary $\partial P_1 \cong \partial Q_1$.
  Since $\partial P_1 \cong \partial Q_1$ is a 4-manifold, homotopy implies isotopy.
  It follows that the attaching curves are isotopic.
  
  It remains to investigate the framing.
  The attaching circles of the 2-handles of $P$ may have framing different from that of~$Q$.
  Thus, the boundary $\partial P$ is obtained from $\partial Q\cong S^4$ by Gluck twists along (some of) the belt spheres of the 2-handles of~$Q$.
  Let $h^2_i$ be the 2-handles of the 3-manifold~$\Sigma(h_0)$, $i=1,\ldots,2g+1$.
  Let $\gamma_i$ be the cocore arc of~$h^2_i$.
  Note that $(\gamma_i,\partial\gamma_i) \subset (\Sigma(h_0), \partial\Sigma(h_0)) = (D^3,S^2)$. 
  So $\partial(\gamma_i \times I) \subset \partial(\Sigma(h_0)\times I)=S^3$ is a ribbon knot of the form $K_i \csum -K_i$, where $K_i\subset S^3$ is the knot obtained by capping off $(D^3,\gamma_i)$ with the standard unknotted pair $(D^3,D^1)$.
  Also, the 2-disk $\gamma_i \times I$ is a (pushed-in) ribbon disk in $\Sigma(h_0)\times I=D^4$ bounded by the ribbon knot $\partial(\gamma_i \times I)$.
  Therefore the sphere $\partial(\gamma_i\times I^2)$ in $\partial(\Sigma(h_0)\times I^2)=S^4$, which is the belt sphere of the 2-handle $h^2_i$, is the double of the ribbon 2-disk $\gamma_i \times I \subset D^4$.
  By~\cite{Melvin:1977-1} (see also \cite[6.2.11(b)]{Gompf-Stipsicz:1999-1}), a Gluck twist on $S^4$ along the double of a ribbon 2-disk (or along a ribbon 2-knot) does not change the diffeomorphism type.
  The proof given in \cite[Solution of 6.2.11(b), p.~495]{Gompf-Stipsicz:1999-1}
  can actually be carried out without modification to show the multi-component version.
  From this, it follows that $\partial P$ is diffeomorphic to~$\partial Q$.
\end{proof}

\bibliographystyle{amsalphanobts}
\def\MR#1{}
\bibliography{research}

\end{document}